\documentclass[12pt,a4paper]{amsart}

\usepackage{amsmath,amssymb,amsthm}
\usepackage{mathrsfs}
\usepackage{geometry}
\usepackage{hyperref}
\usepackage{xcolor}
\usepackage{enumitem}
\usepackage{tcolorbox}
\tcbuselibrary{skins}
\newtheorem{theorem}{Theorem}[section]

\newtheorem{proposition}[theorem]{Proposition}
\newtheorem{corollary}[theorem]{Corollary}
\theoremstyle{definition}
\newtheorem{definition}[theorem]{Definition}
\newtheorem{remark}[theorem]{Remark}

\newcommand{\Lp}[1]{L^{#1}(0,\infty)}

\newcommand{\dt}{\,dt}
\newcommand{\dx}{\,dx}

\tcbset{notebox/.style={colback=blue!4,colframe=blue!40!black,
  arc=2pt,fonttitle=\bfseries,title={Historical Note},left=6pt,right=6pt}}

\title{A sharp logarithmic condition for the Hardy operator
on $\Lp{1}$ and $\ell^1$}
\author {Samson Owusu-Ensaw}
\address{Department of Mathematics, University of Ghana, PO. Box LG 62 Legon, Accra, Ghana }
\email{ sowusu-ensaw@st.ug.edu.gh}
\author {Beno\^it F. Sehba}
\address{Department of Mathematics, University of Ghana, PO. Box LG 62 Legon, Accra, Ghana }
\email{ bfsehba@ug.edu.gh}
\author {Ransford T. Tweneboanah}
\address{Department of Mathematics, University of Ghana, PO. Box LG 62 Legon, Accra, Ghana }
\email{ ranstweneboanah@gmail.com}
\date{}
\begin{document}
\maketitle

\begin{abstract}
%We study the Hardy operator at the endpoint $L^1$. We first show
%that integrability of the image forces the function, or sequence,
%to have mean zero. Motivated by this, we introduce a modified
%Hardy operator and characterize, by a sharp logarithmic
%integrability condition, the largest subspace of %$L^1$ on which
%it maps into $L^1$. A complete discrete analogue is established
%on $\ell^1$, with norm equivalences in both settings.
The Hardy operator is not bounded on the space of integrable
functions on the positive half-line and its discrete counterpart on summable
sequences.
%We study the classical Hardy operator on the space of integrable
%functions on the positive half-line and its discrete counterpart on summable
%sequences. We first show that integrability of the image forces the
%function, or sequence, to have mean zero. Motivated by this necessary
%condition, 
we introduce a modified Hardy operator obtained by
subtracting a natural corrective term, and characterize the largest
subspace of integrable functions on which this modified operator maps
into integrable functions. The sharp condition is a logarithmic
integrability (summability) requirement whose weight reflects obstructions on both
small and large scales. 
%A
%complete discrete analogue is established,
%where the sharp condition takes the form of a logarithmically weighted
%summability condition. In both settings, the norm of the modified
%operator is equivalent to the relevant weighted quantity.
\end{abstract}

%\tableofcontents

% ─────────────────────────────────────────────────────────────────────────────
\section{Introduction}
%\begin{tcolorbox}[notebox]
The operator that bears Hardy's name arose in his 1920 paper
\emph{Note on a theorem of Hilbert} \cite{Hardy1920}, in which Hardy
gave an elegant proof of a result on double series due to Hilbert.
The continuous analogue appeared in explicit form in a 1925 paper
\cite{Hardy1925}, where Hardy proved:
\[
  \int_0^\infty \!\left(\frac{1}{x}\int_0^x f(t)\dt\right)^p\!\dx
  \;\le\;\left(\frac{p}{p-1}\right)^p\int_0^\infty f(x)^p\dx,
  \qquad 1<p<\infty,\; f\ge 0.
\]
The constant $(p/(p-1))^p$ is sharp. Hardy attributed the problem to Landau and Riesz;
the result was refined and extended by Hardy,
Littlewood, and P\'olya in their celebrated monograph \cite{HLP1934}. The above inequality is known as Hardy's inequality, for more on this inequality, its discrete version, and related topics, we refer the interested reader to \cite{Kufner2006,Kufner2007,Muckenhoupt1972,Opic1990}.
%The \emph{dual} or \emph{adjoint} version,
%$\mathcal{Q}f(x)=\frac{1}{x}\int_x^\infty f(t)\dt$,
%satisfies the same $L^p$ bound and is sometimes called
%the \emph{conjugate Hardy operator} or \emph{Copson operator}
%after E.~T.~Copson \cite{Copson1927}, who studied it
%systematically.  Both forms appear throughout the modern theory
%of weighted norm inequalities and interpolation spaces.
%\end{tcolorbox}

The continuous Hardy operator is defined for measurable functions $f$
by
\begin{equation}\label{eq:Qdef}
  \mathcal{Q}f(x) = \frac{1}{x}\int_0^{x} f(t)\dt, \qquad x>0.
\end{equation}
%\end{definition}

We recall that for $1\leq p<\infty$, the Lebesgue space $\Lp{p}$ consists of all measurable functions $f$ on $(0,\infty)$ such that
$$\Vert f\Vert_{L^p}:=\left(\int_0^\infty \vert f(x)\vert^p\dx\right)^{\frac{1}{p}}<\infty.$$
%The key observation is that $\mathcal{Q}f\notin\Lp{1}$ whenever
%$\int_0^\infty f\dt\ne 0$
The operator $\mathcal{Q}$ is not bounded from $\Lp{1}$ into itself. In effect, consider for example the function $$f_0(x)=\frac{\chi_{[1,2]}(x)-\chi_{[3,4]}(x)}{x}.$$ It is easy to see that $f_0\in\Lp{1}$ and that 
\begin{equation*}
\mathcal{Q}f_0(x) = 
\begin{cases} 
      0 & \text{if } 0<x\leq 1 \\
      \frac{\ln x}{x} & \text{if } 1 < x\leq 2 \\
      \frac{\ln 2}{x} & \text{if } 2<x\leq 3 \\ 
      \frac{\ln 6-\ln x}{x} & \text{if } 3<x\leq 4 \\
      \frac{\ln(3/2)}{x} & \text{if } x>4.
\end{cases}
\end{equation*}
It clearly follows that $\mathcal{Q}f_0\notin \Lp{1}$.
In \cite{HL1930}, Hardy and Littlewood made the same observation.
%that $\mathcal{Q}$ is not bounded on $L^1((0,\infty))$. %Indeed, assuming, for example, that $f$ is non-negative, then
%\begin{eqnarray*}
%\int_1^\infty \mathcal{Q}(f)(x)dx &>& \int_1^\infty %\left(\frac{1}{x}\int_{1/n}^x\frac{1}{x}dx\right)dt\\ &=& \int_0^\infty
%\end{eqnarray*}
They then proved that nevertheless, the following local result holds.
\begin{theorem}\label{thm:HL}\cite[Theorem 11]{HL1930}
Let $a>0$. Assume that $f$ is a positive measurable function defined on $(0,\infty)$. Then the following hold.
\begin{itemize}
    \item[(i)] If $\int_0^a f(x)\log^+f(x)dx<\infty$, then $$\int_0^a\mathcal{Q}f(x)dx\lesssim 1+\int_0^a f(x)\log^+f(x)dx.$$
    \item[(ii)] If $f$ is decreasing and $K:=\int_0^a\mathcal{Q}f(x)dx<\infty$, then
    $$\int_0^a f(x)\log^+f(x)dx\lesssim 1+K\log^+K.$$
\end{itemize}
\end{theorem}
We recall that for $x>0$, $\log^+x=\max\{0,\log x\}$. The notation $\log$ is used for function while $\ln$ is used when the argument is a number.
\vskip .2cm
We aim in this note to consider a discussion of the case where the function under consideration is supported throughout the interval $(0,\infty)$ and the integration of its transform is also considered throughout $(0,\infty)$. We also discuss its discrete analogue.
\vskip .2cm
To be more precise on the problem in consideration, knowing that $\mathcal{Q}$ is not bounded from $L^1((0,\infty))$ to itself, we ask the following question: what is the maximal subspace of $\Lp{1}$ in which we can choose $f$ so that $\mathcal{Q}f\in \Lp{1}$.
\vskip .2cm

Our first observation will be that for an integrable function $f$,  $\mathcal{Q}f\in \Lp{1}$ only if the integral of $f$ is zero. This means in particular that for a non-negative integrable function $f$, we cannot have $ \mathcal{Q}f\in \Lp{1}$. Thus, for the problem to still be considered for functions $f$ in $\Lp{1}$, there is a need to perturb $f$ to obtain a function $\tilde{f}$ with the properties: $\tilde{f}\in \Lp{1}$ and $\int_0^\infty \tilde{f}(x)dx=0$. We will see that this is equivalent (for an appropriate perturbation) to considering the boundedness on $ \Lp{1}$ of the operator
$$\mathcal{H}f(x)=\mathcal{Q}f(x)-\frac{1}{x+1}\int_0^\infty f(t)dt.$$
Our result can be stated as follows: 

If $f$ is non negative and integrable, then $\mathcal{H}f\in L^1((0,\infty))$ if and only if $$\int_0^\infty f(x)\ln\left(2+x+\frac{1}{x}\right)dx<\infty,$$
with equivalence of the norms.
%We will also prove the discrete analog of the above result.
\vskip .2cm
One may ask if there is any relation between the Hardy and Littlewood local result and ours. The answer is quite simple: their result is not only local but also, we are not considering the same operator. Also, because of the necessity for an integrable function to have mean zero for its image by the Hardy's operator to be in $\Lp{1}$, one cannot expect Theorem \ref{thm:HL} to be global.
\vskip .2cm
Let us say a few words about the discrete results. The discrete Hardy
operator $\Gamma$ acts on a sequence $a = (a_k)_{k \geq 1}$ of real
numbers by forming its running Ces\`{a}ro means,
\[
  (\Gamma a)_n = \frac{1}{n}\sum_{k=1}^n a_k, \qquad n \geq 1.
\]
The operator $\Gamma$ is bounded on $\ell^p$ (the space of $p$-summable sequences) for every $1 < p < \infty$ with
the sharp constant $(p/(p-1))^p$ (the discrete counterpart of Hardy's
classical inequality). As in the continuous case, the boundedness at the endpoint $p = 1$ fails: the sequence
$(1/(k(k+1)))_k$ is summable, yet its image under $\Gamma$ has general
term $1/(n+1)$, which is not summable. The first discrete result of
this paper shows that this failure is
not an accident of the example: if a summable sequence $a$ has a
summable image $\Gamma a$, then the sum of $a$ must be zero. 
%The proof
%mirrors the continuous argument, using the fact that the partial sums
%of $a$ converge to the total sum, so that $|(\Gamma a)_n|$ is
%eventually bounded below by a positive multiple of $1/n$, which is notsummable. 
This necessary condition motivates the introduction of the
discrete modified operator
\[
  (\widetilde{\Gamma} a)_n
  = \frac{1}{n}\sum_{k=1}^n a_k - \frac{1}{n+1}\sum_{k=1}^\infty a_k,
  \qquad n \geq 1,
\]
obtained by subtracting the corrective term
$\frac{1}{n+1}\sum_{k=1}^\infty a_k$, in exact analogy with the
continuous construction. The discrete main result
identifies the sharp condition under which
$\widetilde{\Gamma}a \in \ell^1$: for a sequence of positive terms,
this holds if and only if $\sum_{k=1}^\infty a_k \ln(k+1) < \infty$.
The logarithmic weight $\ln(k+1)$ is the discrete counterpart of the
weight $\ln(1+t)$ that governs the large-scale behavior in the
continuous result; note that, unlike the continuous case, there is no
small-scale logarithmic obstruction for sequences, since the index $k$
is bounded away from zero. The proof relies on a splitting of
$\widetilde{\Gamma}a$ into two parts, analogous to the decomposition
used in the continuous case, and on a comparison between partial
harmonic sums and their logarithmic asymptotic via the
Euler--Mascheroni constant. As a consequence, the $\ell^1$-norm of
$\widetilde{\Gamma}a$ is equivalent to
$\gamma \sum_k a_k + \sum_k a_k \ln(k+1)$, providing a quantitative
version of the characterization.
\vskip .2cm
We are motivated not only by Hardy-Littlewood local result Theorem \ref{thm:HL}, but also by some recent and old results of the same vein obtained for some other operators: the Hilbert's integral operator, the Bergman projection, and the maximal functions (see \cite{BGS2024,BGS2023,BGS2022,BGS2024AB,Stein1969}).
\vskip .2cm
In the next section, we study the problem for the continuous Hardy operator while in Section 3, we discuss the discrete analogue results. 
\vskip .2cm
Throughout this text, for two positive quantities $A$ and $B$, the notation $A\lesssim B$ (resp. $A\gtrsim B$) means that $A\le CB$ ($A\ge CB$) for some universal constant $C>0$. If both $A\lesssim B$ and $B\lesssim A$ hold, we write $A\approx B$.
\section{The result for the continuous Hardy operator}
% ─────────────────────────────────────────────────────────────────────────────

The key observation is that $\mathcal{Q}f\notin L^{1}((0,\infty))$ whenever
$\int_0^\infty f\dt\ne 0$, because the integrand behaves like a
positive constant divided by $x$ near infinity.  The following
proposition makes this precise.

\begin{proposition}\label{prop:necessary}
Suppose $f\in\Lp{1}$ and $\mathcal{Q}f\in\Lp{1}$.
Then
\begin{equation}\label{eq:zero}
  \int_0^{\infty} f(t)\dt = 0.
\end{equation}
\end{proposition}

\begin{proof}
For $x>0$, multiply both sides of~\eqref{eq:Qdef} by $x$:
\begin{equation}\label{eq:xQ}
  x\,\mathcal{Q}f(x) = \int_0^{x} f(t)\dt.
\end{equation}
Since $f\in L^{1}((0,\infty))$, the right-hand side of~\eqref{eq:xQ}
converges to $\int_0^\infty f(t)\dt$ as $x\to \infty$.  In particular,
\begin{equation}\label{eq:limit}
  \lim_{x\to \infty} x\,\vert\mathcal{Q}f(x)\vert = \left\vert\int_0^{\infty} f(t)\dt\right\vert =: \ell.
\end{equation}

Suppose for contradiction that $\ell\ne 0$; without loss of generality
assume $\ell = 1$ (replace $f$ by $\frac{f}{\ell}$ if necessary).
By~\eqref{eq:limit} for $\varepsilon = \tfrac{1}{2}$, there exists
$M>0$ such that for all $x>M$,
\[
  \bigl|x\,\vert\mathcal{Q}f(x)\vert - 1\bigr| < \tfrac{1}{2},
  \qquad\text{i.e.}\qquad
  x\,\vert\mathcal{Q}f(x)\vert > \tfrac{1}{2} .
\]
It follows that for all $x>M$,
\begin{equation}\label{eq:lbnd}
  \mathcal{Q}f(x) > \frac{1}{2x},
\end{equation}
whence
\[
  \int_0^\infty \vert\mathcal{Q}f(x)\vert\dx \ge \frac{1}{2}\int_M^\infty \frac{\dx}{x} = \infty.
\]
This contradicts $\mathcal{Q}f\in\Lp{1}$.  Hence, necessarily $\ell=0$,
which is~\eqref{eq:zero}.
\end{proof}

\begin{remark}
The condition $\int_0^\infty f(t)\dt=0$ is \emph{necessary} but not
sufficient for $\mathcal{Q}f\in\Lp{1}$. For instance,
$f_e(t)=\frac{1}{t(\ln t)^2}\left[\mathbf{1}_{(0,e^{-1})}(t)-\mathbf{1}_{(e,\infty)}(t)\right]$ has integral zero,
but \begin{equation*}
\mathcal{Q}f_e(x) = 
\begin{cases} 
      -\frac{1}{x\ln x} & \text{if } 0<x\leq e^{-1} \\
      \frac{1}{x}  & \text{if }  e^{-1}<x<e \\
      \frac{1}{x\ln x} & \text{if } x> e
\end{cases}
\end{equation*}
is not integrable.
\end{remark}

It follows from the above proposition that the integrability of $\mathcal{Q}f$ can only be considered for functions $f\in L^1((0,\infty))$ such that $\int_0^\infty f(t)\dt=0$. An alternative being that one can replace $f$ by $f-\theta\int_0^\infty f(t)\dt$, where $\theta$ is a smooth function with unit integral.
\vskip .2cm
Consider $\theta(t)=\tfrac{1}{(1+t)^2}$ which is smooth on $(0,\infty)$. Then we have $$\mathcal{Q}\theta(x)=\frac{1}{x}\int_0^x\theta(t)dt=\frac{1}{x}\int_0^x\frac{dt}{(1+t)^2}=\frac{1}{1+x},$$
and $$\int_0^\infty\theta(t)dt=1.$$
If we choose $\theta$ as above, then we are led to study the following operator in place of $\mathcal{Q}$:
\begin{equation}\label{eq:pertubHardy}
    \mathcal{H}f(x)=\mathcal{Q}\left(f-\theta\left(\int_0^\infty f(t)dt\right)\right)(x)=\frac{1}{x}\int_0^x f(t)dt-\frac{1}{1+x}\int_0^\infty f(t)dt, 
\end{equation}
for $f\in L^1((0,\infty))$.
\vskip .2cm
Our main result of this section is the following.
\begin{theorem}\label{thm:main1}
Let $f\in L^1((0,\infty))$. Then $\mathcal{H}f\in L^1((0,\infty))$ if $f$ satisfies
\begin{equation}\label{eq:loginteg}
\int_0^\infty\vert f(t)\vert\left[\ln\left(1+\frac{1}{t}\right)+\ln\left(1+t\right)\right]<\infty.
\end{equation}
Conversely, if $f$ is non-negative, then $\mathcal{H}f\in L^1((0,\infty))$ implies that $f$ satisfies (\ref{eq:loginteg}).    
\end{theorem}
Obviously, we have $\ln\left(1+\frac{1}{t}\right)+\ln\left(1+t\right)=\ln\left(2+t+\frac{1}{t}\right)$ but we choose to write it as above to mark the fact that only one of the logarithms matters whenever we are dealing with small or large variables. We also note that this weight has the same behavior near zero and at infinity.
\begin{proof}[Proof of Theorem \ref{thm:main1}]
    \textbf{Sufficiency.}
Assume $f\in\Lp{1}$ satisfies (\ref{eq:loginteg}).
Using the decomposition
\begin{equation}\label{eq:Hsplit}
    \mathcal{H}f(x)=\left(\frac{1}{x}-\frac{1}{x+1}\right)\int_0^x f(t)dt-\frac{1}{x+1}\int_x^\infty f(t)dt,
\end{equation}
%~\eqref{eq:Hsplit} 
we bound
\begin{eqnarray*}\label{eq:Hest}
  \int_0^\infty |\mathcal{H}f(x)|\dx
  &\le&
  \int_0^\infty \left(\frac{1}{x}-\frac{1}{x+1}\right)\int_0^x|f(t)|\dt\,\dx
  +
  \int_0^\infty \frac{1}{x+1}\int_x^\infty|f(t)|\dt\,\dx\\ &=& I_1+I_2
\end{eqnarray*}
where $$I_1:=\int_0^\infty \left(\frac{1}{x}-\frac{1}{x+1}\right)\int_0^x|f(t)|\dt\,\dx$$
and $$I_2:=\int_0^\infty \frac{1}{x+1}\int_x^\infty|f(t)|\dt\,\dx.$$
\medskip\noindent
\textit{Estimate of $I_1$.}
By Fubini--Tonelli (exchanging the order of integration over the region
$0<t<x<\infty$),
\begin{align}
  I_1
  &= \int_0^\infty |f(t)|\int_t^\infty
     \!\left(\frac{1}{x}-\frac{1}{x+1}\right)\dx\,\dt \notag\\
  &= \int_0^\infty |f(t)|\,\ln\!\left(\frac{t+1}{t}\right)\dt
  \;<\;\infty. 
\end{align}

\medskip\noindent
\textit{Estimate of $I_2$.}
By Fubini--Tonelli (region $0<x<t<\infty$),
\begin{align}
  I_2
  &= \int_0^\infty |f(t)|\int_0^t \frac{\dx}{x+1}\dt \notag\\
  &= \int_0^\infty |f(t)|\,\ln(t+1)\dt
  \;<\;\infty.
\end{align}

Hence $\mathcal{H}f\in\Lp{1}$.

\medskip
\textbf{Necessity.}
Conversely, assume $f$ is a non-negative integrable function such that $\mathcal{H}f\in\Lp{1}$. As $f\in\Lp{1}$, we have
\begin{eqnarray*}
\int_1^\infty\left[\left(\frac{1}{x}-\frac{1}{x+1}\right)\int_0^x f(t)dt\right]dx &=&  \int_0^\infty f(t)\left[\int_{\max\{1,t\}}^\infty\left(\frac{1}{x}-\frac{1}{x+1}\right)dx\right]dt\\ &\leq&  \int_0^\infty f(t)\left[\int_1^\infty\left(\frac{1}{x}-\frac{1}{x+1}\right)dx\right]dt\\ &=& \int_0^\infty f(t)\left(\ln(2)\right)dt\\ &<& \infty.
\end{eqnarray*}
We deduce that
$$\left\vert\int_1^\infty\left[\mathcal{H}f(x)-\left(\frac{1}{x}-\frac{1}{x+1}\right)\int_0^x f(t)dt\right]dx\right\vert<\infty.$$
Hence
$$\int_1^\infty\frac{1}{x+1}\int_x^\infty f(t)dtdx<\infty.$$
That is $$\infty>\int_1^\infty f(t)\left(\int_1^t\frac{1}{x+1}dx\right)dt=\int_1^\infty f(t)\left[\ln(1+t)-\ln2\right]dt.$$
Thus $$\int_0^\infty f(t)\ln(1+t)dt<\infty.$$
It follows that
\begin{eqnarray*}\int_0^\infty\left(\frac{1}{x+1}\int_x^\infty f(t)dt\right)dx &=& \int_0^\infty f(t)\left(\int_0^t\frac{1}{x+1}dx\right)dt\\ &=&\int_0^\infty f(t)\ln(1+t)dt\\ &<&\infty.
\end{eqnarray*}
We deduce that
$$\int_0^\infty \left[\mathcal{H}f(x)+\frac{1}{x+1}\int_x^\infty f(t)dt\right]dx<\infty.$$
Hence $$\int_0^\infty\left[\left(\frac{1}{x}-\frac{1}{x+1}\right)\int_0^x f(t)dt\right]dx<\infty.$$
That is
$$\int_0^\infty f(t)\left[\int_t^\infty\left(\frac{1}{x}-\frac{1}{x+1}\right)dx\right]dt<\infty.$$
Thus $$\int_0^\infty f(t)\ln\left(1+\frac{1}{t}\right)dt<\infty.$$
The proof is complete.
\end{proof}
In particular, we have the following.
\begin{corollary}\label{cor:main1}
Let $f\in \Lp{1}$ be non negative. Then $\mathcal{H}f\in \Lp{1}$ if and only  $f$ satisfies (\ref{eq:loginteg}). Moreover, 
\begin{equation}\label{eq:equivnorm}
\Vert \mathcal{H}f\Vert_{L^1}\approx\int_0^\infty f(t)\left[\ln\left(1+\frac{1}{t}\right)+\ln\left(1+t\right)\right]dt.
\end{equation}
%Conversely, if $f$ is non-negative, then $\mathcal{H}f\in L^1((0,\infty))$ implies that $f$ satisfies (\ref{eq:loginteg}).    
\end{corollary}

% ─────────────────────────────────────────────────────────────────────────────
\section{The case of the discrete Hardy operator}
% ─────────────────────────────────────────────────────────────────────────────

The Hardy operator has a natural discrete counterpart that predates
the continuous version and retains independent importance in sequence
spaces and summability theory. We recall the following definition of sequence spaces.
\begin{definition}\label{def:sequencespace}
For $1\leq p<\infty$, 
    $$\ell^p:=\ell^{p}(\mathbb{N})=\left\{(a_{k})_{k\in\mathbb{N}} : \Vert (a_{k})_k\Vert_{\ell^p}^p:=\sum_{k=1}^{\infty}|a_{k}|^p<\infty\right\}.$$
    %\quad $\|a\|_{l^{1}}=\sum_{k=1}^{\infty}|a_{k}|.%$
\end{definition}

%\subsection{Definition and basic properties}

Let $a=(a_n)_{n=1}^\infty$ be a sequence of reals.
The \emph{discrete Hardy operator} is
\begin{equation}\label{eq:disH}
  (\Gamma a)_n =A_n:= \frac{1}{n}\sum_{k=1}^n a_k, \qquad n\ge 1,
\end{equation}
the running Ces\`aro mean of the sequence.
%The \emph{dual} (or \emph{Copson}) discrete operator is
%\begin{equation}\label{eq:disQ}
%  (Q a)_n = \sum_{k=n}^\infty \frac{a_k}{k}, \qquad n\ge 1.
%\end{equation}

%\subsection{The discrete Hardy inequality}
The boundedness of $\Gamma$ on the discrete space $\ell^p(\mathbb{N})$ ($1<p<\infty$) is a classic in the literature.
\begin{theorem}[Hardy, 1920]\label{thm:disHardy}
For $1<p<\infty$ and $(a_n)\in\ell^p$ with $a_n\ge 0$,
\begin{equation}\label{eq:discineq}
  \sum_{n=1}^\infty \left(\frac{1}{n}\sum_{k=1}^n a_k\right)^p
  \;\le\;
  \left(\frac{p}{p-1}\right)^p \sum_{n=1}^\infty a_n^p.
\end{equation}
The constant $(p/(p-1))^p$ is sharp.
\end{theorem}

Inequality~\eqref{eq:discineq} fails when $p=1$, as constant 
$p/(p-1)\to\infty$. In fact, it is easy to see that in this limit case, $\Gamma$ is not bounded: take $a_k=\frac{1}{k(k+1)}$. Then $\Vert (a_n)_n\Vert_{\ell^1}=1$ while the sequence with general term $\left(\Gamma a\right)_n=\frac{1}{n}\sum_{k=1}^n\frac{1}{k(k+1)}=\frac{1}{n+1}$ is not in $\ell^1$. 

The discrete analogue of Proposition~\ref{prop:necessary}
is the following.

\begin{proposition}\label{prop:disNec}
Let $a=(a_n)\in\ell^1$.  If $\Gamma(a) \in \ell^1$,
then $\displaystyle\sum_{k=1}^\infty a_k = 0$.
\end{proposition}

\begin{proof}
Suppose for contradiction that $\lambda:=\displaystyle\sum_{k=1}^{\infty}a_{k}\ne 0$. Since $nA_{n}=\displaystyle\sum_{k=1}^{n}a_{k}\rightarrow\lambda$ as $n\to\infty$, we have $\displaystyle\lim_{n\rightarrow\infty}n|A_{n}|=|\lambda|$. Thus, for $\epsilon=\frac{|\lambda|}{2}$, there exists $N\in\mathbb{N}$ such that for all integers $n\ge N$: $\left|n|A_{n}|-|\lambda|\right|<\frac{|\lambda|}{2}$. This gives $|A_{n}|>\frac{|\lambda|}{2n}$ for all $n\geq N$. Therefore,
    \begin{equation}
        \displaystyle\sum_{n=1}^{\infty}|A_{n}|>\frac{|\lambda|}{2}\sum_{n=N}^{\infty}\frac{1}{n}=\infty
    \end{equation}
contradicting $\Gamma (a)\in\ell^1$.
\end{proof}

In analogy with the continuous operator $\mathcal{H}$, if we let $\lambda=(\lambda_n)_n$ be the sequence defined by $\lambda_n=\frac{1}{n(n+1)}$, then the \emph{discrete modified operator} is defined as
\begin{equation}\label{eq:discH}
  (\tilde{\Gamma}a)_n
  =\left(\Gamma(a-\lambda)\right)_n= \frac{1}{n}\sum_{k=1}^n a_k
    - \frac{1}{n+1}\sum_{k=1}^\infty a_k,
  \qquad n\ge 1.
\end{equation}
The following is the discrete analogue of Theorem~\ref{thm:main1}.

\begin{theorem}\label{thm:main2}
Let $a=(a_n)\in\ell^1$.  Then $(\tilde{\Gamma}a)\in\ell^1$ if 
\begin{equation}\label{eq:disclogcond}
  \sum_{k=1}^\infty |a_k|\ln(k+1) < \infty.
\end{equation}
Conversely, if $a=(a_n)_n$ is sequence of positive terms and $(\tilde{\Gamma}a)\in\ell^1$, then $a=(a_n)_n$ satisfies (\ref{eq:disclogcond}).
\end{theorem}

\begin{proof}
Decompose $(\tilde{\Gamma}a)_n = J_1(n) - J_2(n)$ where
\begin{equation}\label{eq:gammasplit}
  J_1(n) = \left(\frac{1}{n}-\frac{1}{n+1}\right)\sum_{k=1}^n a_k,
  \qquad
  J_2(n) = \frac{1}{n+1}\sum_{k=n+1}^\infty a_k.
\end{equation}
For the $\ell^1$-norm of $J_1$, we easily obtain
%and the equivalence 
%$$\sum_{k=1}^n\frac{1}{n}-\ln(n)\simeq \gamma$$
%where $\gamma$ is the Euler-Mascheroni's constant,
we obtain
\[
  \sum_{n=1}^\infty |J_1(n)|
  \le \sum_{k=1}^\infty |a_k|\sum_{n=k}^\infty \left(\frac{1}{n}-\frac{1}{n+1}\right)
  = \sum_{k=1}^\infty|a_k|\cdot\frac{1}{k},
\]
which is finite since $\sum_{k=1}^\infty |a_k|<\infty$.

Using Fubini's theorem and the equivalence
$$\sum_{k=1}^n\frac{1}{n}-\ln(n)\approx \gamma$$
where $\gamma$ is the Euler-Mascheroni's constant, we obtain
\begin{eqnarray*}
  \sum_{n=1}^\infty|J_2(n)|
  &\le& \sum_{k=2}^\infty |a_k|\sum_{n=1}^{k-1}\frac{1}{n+1}\\
 &=& \sum_{k=1}^\infty |a_k|\left[\left(\sum_{n=1}^{k}\frac{1}{n}-\ln k\right)+\ln(k+1)\right]\\ &\lesssim& \gamma\left(\sum_{k=1}^\infty |a_k|\right)+\sum_{k=1}^\infty |a_k|\ln(k+1)<\infty.
\end{eqnarray*}
%where $\mathcal{H}_k=\sum_{j=1}^k 1/j\sim\ln k$ is the $k$-th harmonic
%number.  The logarithmic condition~\eqref{eq:disclogcond} is therefore
%equivalent to $\sum|a_k|\mathcal{H}_k<\infty$, completing the sufficiency.

Conversely, assume $a=(a_n)_n$ is such that $a_k>0$ for all $k\in\mathbb{N}$, and that $\tilde{\Gamma}a\in\ell^1$. We still write $\tilde{\Gamma}a= J_1+J_2$ where $J_1$ and $J_2$ are defined by (\ref{eq:gammasplit}). We have seen that the $\ell^1$-norm of $J_1$ is controlled by $\sum_{k=1}^\infty a_k$. Thus, we necessarily have that the sequence $J_2$ belongs to $\ell^1$. It follows that
\begin{eqnarray*}
    \infty &>& \sum_{n=1}^\infty\frac{1}{n+1}\sum_{k=n+1}^\infty a_k\\ &=& \sum_{k=1}^\infty a_k\sum_{n=1}^{k-1}\frac{1}{n+1}\\ &>& \sum_{k=3}^\infty a_k\sum_{n=2}^{k-1}\frac{1}{n}\\ &=& \sum_{k=3}^\infty a_k\left[\left(\sum_{n=2}^{k-1}\frac{1}{n}-\ln(k-1)\right)+\ln(k-1)\right]\\ &\gtrsim& \gamma\sum_{k=3}^\infty a_k+\sum_{k=3}^\infty a_k\ln(k-1).
\end{eqnarray*}
Thus $\displaystyle\sum_{k=3}^\infty a_k\ln(k-1)<\infty$. This implies that $\displaystyle\sum_{k=1}^\infty a_k\ln(k+1)<\infty$. The proof is complete.
\end{proof}
In particular, we have the following.
\begin{corollary}\label{cor:main2}
Let $a=(a_n)\in\ell^1$ be a sequence of positive reals.  Then $(\tilde{\Gamma}a)\in\ell^1$ if and only if
\begin{equation*}
  \sum_{k=1}^\infty |a_k|\ln(k+1) < \infty.
\end{equation*}
Moreover,
$$\Vert\tilde{\Gamma}a\Vert_{\ell^1}\approx \gamma\sum_{k=1}^\infty a_k+\sum_{k=1}^\infty a_k\ln(k+1).$$
\end{corollary}
\section{Conclusion}
We have shown that the failure of the Hardy operator to be bounded
on $L^1(0,\infty)$ is not merely a defect of the endpoint, but
carries a precise structural obstruction: the integrability of the image requires that the function have mean zero. Once this is taken
into account via a natural corrective perturbation, the resulting modified
operator $\mathcal{H}$ admits a sharp characterization of its
domain of integrability in terms of the logarithmic weight
$\ln(2+t+1/t)$, which simultaneously captures the behavior near
the origin and at infinity. The discrete analogue is fully parallel,
with the single logarithmic weight $\ln(k+1)$ playing the
corresponding role, the absence of a small-scale obstruction
reflecting the discreteness of the index. Both results fit into
a broader program of identifying sharp substitutes for
$L^1$-boundedness of classical operators, alongside analogous
results for the Hilbert-Hardy operator, the Bergman projection, and
maximal functions. Finally, it would be interesting to consider the case of Hardy operator in a higher dimension and other operators of the literature that are bounded on $L^p$ for $1<p<\infty$, but not on $L^{1}$.

\medskip
\section{Compliance with Ethical Standards}
\begin{itemize}
\item {\bf Funding}

There is no funding support to declare.

\item {\bf Disclosure of potential conflicts of interest}

The authors have no relevant financial or non-financial interests to disclose.

\item {\bf Author Contributions}

All authors contributed to the conception and design of the study.  All authors read and approved the final manuscript.

\item {\bf Data availability statements}

Data sharing is not applicable to this article, as no data was created or analyzed in this study.\\

%\item{\bf Use of AI tools}

%The authors declare that they did not use Artificial Intelligence (AI ) tools in the creation of this paper.
\end{itemize}

\end{document}